\documentclass{article}
\usepackage{amsmath}
\usepackage{amsfonts}
\usepackage{amsthm}
\usepackage{amssymb}
\usepackage{fullpage}
\usepackage{tikz-cd}
\usepackage{graphicx}
\usepackage{subcaption}
\providecommand{\U}[1]{\protect\rule{.1in}{.1in}}
\newtheorem{theorem}{Theorem}

\newtheorem{definition}[theorem]{Definition}
\newtheorem{example}[theorem]{Example}

\newtheorem{lemma}[theorem]{Lemma}

\newtheorem{problem}[theorem]{Problem}

\allowdisplaybreaks
\begin{document}

\title{Almost Bijective Parametrization of \\ 
			Copositive Univariate Polynomials }
\author{Hoon Hong and Ezra Nance}
\maketitle

\begin{abstract}
In this work we develop a novel recursive method for parametrizing the cone of copositive univariate polynomials of any arbitrary degree $d$. This parametrization is surjective, almost injective, and has the easily described domain of $(\mathbb{R}_{\geq 0})^d$.

\end{abstract}

\section{Introduction \& Background}

A polynomial whose value remains positive for all inputs is referred to as
\textit{positive}. Given that this is a useful property for a polynomial to
have, these objects have been the subject of much study over the years
[References]. In this paper we will loosen the requirement that a polynomial
must always be positive, and instead we will consider polynomials whose value
is positive for all positive inputs. Polynomials which have this property are
called \textit{copositive}, and are often desirable for applications where
negative values for inputs and outputs don't make sense \cite{chakrabortty2014copositive}, \cite{chen2018copositive}, \cite{kannike2012vacuum}.

Because we can express the set of copositive polynomials using a finite number
quantifiers and inequalities, we know via Tarski \cite{tarski1951decision} that this is a
semialgebraic set. A \textit{semialgebraic} set being a subset of
$\mathbb{R}^{n}$ which satisfies a finite number of inequalities. Indeed,
utilizing the Sturm sequence we could explicitly find a list of inequalities
which the coefficients of a polynomial must satisfy in order to be copositive.
However, this yields an implicit description of the set. In other words, using
these inequalities we can easily check whether a particular polynomial is
copositive or not, but we cannot use this description to generate the set of
all copositive polynomials. This is our primary focus.

We wish to parameterize the set of copositive univariate polynomials. In \cite{gonzalez1996parametrization} it was shown that using the Cylindrical Algebraic
Decomposition it is possible obtain a semialgebraic, bijective
parametrization of any semialgebraic set. However, this was a very general
result, and in our particular case that approach would be very difficult to
implement. As a result we will take a different approach to solve this problem.


\section{Problem}

Parametrization consists of two main components, the domain and the map, and often the complexity of these two components are inversely related. If we start with a very complicated domain, we could parametrize a complicated set with a relatively simple map and vice versa. In the extreme case one can bijectively parametrize any set by simply starting with the set itself as the domain and using the identity map. I hope everyone can agree that this is not a very satisfying or useful parametrization. A ``good" parametrization tends to balance these components. With that in mind we will formally introduce the main property discussed in this paper.

\begin{definition}
	[Copositive polynomial]We say that $f\in\mathbb{R}[x]$ is \emph{copositive},
	and write $f\succeq0$, if
	\[
	\underset{x\geq0}{\forall}\ f\left(  x\right)  \geq0.
	\]
\end{definition}

While copositive polynomials do form a closed convex cone, it is infinite dimensional. So we will consider polynomials of a particular degree. Furthermore, we will only consider monic polynomials. This is illustrated in the following definitions.

\begin{definition}
	[Monic Polynomial Set] The set of monic polynomials of degree $d$ will be written as
	\[F_d = \{f \in \mathbb{R}[x] : deg(f) = d \text{ and } \operatorname{lc}(f) = 1\}. \]
\end{definition}

\begin{definition}
	[Copositive Set]The $d$-\emph{copositive set}, written as $\mathcal{C}_{d}$,
	is defined to be the set of all monic copositive polynomials of degree $d,$
	that is,%
	\[
	\mathcal{C}_{d}=\left\{  f\in F_{d}:\ f\succeq 0 \right\}  .
	\]
\end{definition}
\noindent Here we can visualize a the small dimensional space of $\mathcal{C}_2$.

\begin{example}
[Visualizing $\mathcal{C}_2$] $\mathcal{C}_2 = \{x^2 + a_1x + a_0 : x^2 + a_1x + a_0 \succeq 0 \}$

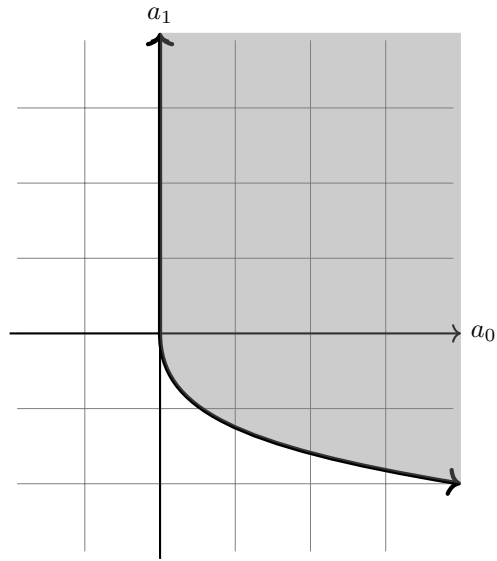
\begin{figure}[h!]
\begin{center}
	\begin{tikzpicture}[domain=0:2]
		\draw[very thin,color=gray] (-1.9,-2.9) grid (3.9,3.9);
		
		\draw[thick,->] (-2,0) -- (4,0) node[right] {$a_0$};
		\draw[thick,->] (0,-3) -- (0,4) node[above] {$a_1$};
		
		\draw[ultra thick,->]    (0,0) .. controls (0,-1) and (1,-1.5).. (4,-2);
		\draw[ultra thick,->]    (0,0) -- (0,4);
		
		\draw [fill=gray,fill opacity=0.4, draw=none] (0,4) -- (0,0).. controls (0,-1) and (1,-1.5).. (4,-2) -- (4,4) -- cycle;
	\end{tikzpicture}
\end{center}
\caption{The set $\mathcal{C}_2$}
\label{fig:C2}
\end{figure}
\end{example}

It's not difficult to show that for each $d \geq 0$, $\mathcal{C}_d$ is a closed convex cone. We will parametrize each of these sets. For the domain we will choose the simplest closed convex cone we can think of, that is, $\mathbb{R}^d_{\geq 0}$. This takes care of the domain. The other main component of a parametrization is the map itself. For this we are going to demand that the map be \emph{polynomial} and \emph{generically bijective}. We can now describe the problem statement.

\begin{problem}
[Main Problem] Given $d \geq 0$, find some polynomial map $\Phi_{d}: \mathbb{R}^d_{\geq 0} \to \mathcal{C}_d$ which is generically bijective.
\end{problem}

At this point it's not clear that such a map should even exist, let alone be found. The next section will step through the logic of constructing such a map.


\section{Main Results}

We will begin by doing a bit of exploration on the small example of $\mathcal{C}_2$. From Figure \ref{fig:C2} we can see that the boundary of $\mathcal{C}_2$ seems to be made of two pieces: one being a vertical ray, and one being the curved ``bottom" piece. Interestingly, if we vertically extend this curved bottom boundary, we can recover all of $\mathcal{C}_2$. So now a natural question arises. Can we characterize this curved boundary? 

Luckily, $\mathcal{C}_2$ is a small enough example that we can basically answer this question by inspection. It turns out that a polynomial $f \in \mathcal{C}_2$ is on this curved boundary if and only if there exists some $x \geq 0$ such that $f(x) = f'(x) = 0$. This gives us an idea for how to form a general approach to the main problem. First we will parametrize the polynomials in $\mathcal{C}_d$ with a non-negative zero which is also a zero of the derivative. Then we will extend these polynomials in some way to recover all of $\mathcal{C}_d$.

With this in mind we will now state the main theorem in this paper. 

\begin{theorem}[Main Result] 
\label{thm:main}
The map $\Phi_{d}$ is a generically bijective parametrization of the
$d$-\emph{copositive set} $\mathcal{C}_{d}$. 
\[
\phi_{d}:%
\begin{array}
[c]{ccc}%
(\mathbb{R}_{\geq0})^{d} & \longrightarrow & \mathbb{R}\left[  x\right] \\
\mathbf{t}_{d} & \longmapsto & f
\end{array}
\]
where $\mathbf{t}_{i}\ $is a short-hand notation for $\left(  t_{0}%
,\ldots,t_{i-1}\right) \in (\mathbb{R}_{\geq 0})^d$ and
\[
f=\Phi_{d}(\mathbf{t}_{d})=\left\{
\begin{array}
[c]{lll}%
1 & \text{if} & d=0\\
x+t_{0} & \text{if} & d=1\\
\left(  x-t_{d-2}\right)  ^{2}\ \Phi_{d-2}(\mathbf{t}_{d-2})+t_{d-1}x &
\text{if} & d\geq2
\end{array}.
\right.
\]

\end{theorem}

\begin{example}
Here we will parametrize $\mathcal{C}_3$.
\begin{align*}
	\Phi_3(t_0,t_1,t_2) &= (x - t_1)^2 \Phi_1(t_0) + t_2x \\
	&= (x^2-2t_1x+t_1^2)(x+t_0) + t_2x \\
	&= x^3 + (t_0 - 2t_1)x^2 + (t_1^2 - 2t_0t_1 + t_2)x + (t_0t_1^2)
\end{align*}
So \[\mathcal{C}_3 = \{x^3 + (t_0 - 2t_1)x^2 + (t_1^2 - 2t_0t_1 + t_2)x + (t_0t_1^2) : t_0,t_1,t_2 \geq 0 \}.\]
\end{example}

As previously stated, the parametrization in Theorem \ref{thm:main} consists of two main steps:
\begin{enumerate}
	\item Parametrize the set $\left\{f \in \mathcal{C}_d : \underset{x \geq 0}{\exists} f(x)=f'(x)=0\right\}$.
	\item Extend the set from step 1 to recover all of $\mathcal{C}_d$.
\end{enumerate}
Using step 1 as motivation we will introduce the following definition.
\begin{definition}[Base Boundary]
	The \emph{base boundary} of the $d$-copositive set, $\mathcal{C}_d$, is defined to be
	\[
		\mathcal{B}_d = \left\{f \in \mathcal{C}_d : \underset{x \geq 0}{\exists}f(x) = f'(x) = 0\right\}.
	\]
\end{definition}
\noindent We will now step through the details of these two steps in the next sections, and we will prove Theorem \ref{thm:main} in section \ref{sec:combo}.

\section{Parametrizing the Base Boundary}
\label{sec:BB}
In this section it will be helpful to have a running example polynomial to work with. So, consider $f \in \mathcal{C}_d$. In terms of its complex roots we have
\[
	f = (x-z_1)(x-z_2)\cdots(x-z_d).
\]
Now if we assume $f$ is in the base boundary, then there is some $t \geq 0$ for which $f(t)=f'(t) = 0$. In this case $f$ takes the form
\[
	f = (x-z_1)(x-z_2)\cdots(x-z_{d-2})(x-t)^2.
\]
Now consider the polynomial
\[
	g = \frac{f}{(x-t)^2} = (x-z_1)(x-z_2)\cdots(x-z_{d-2}).
\]
Note that $g$ must be an element of $\mathcal{C}_{d-2}$. 

Now if we will use this idea to build $\mathcal{B}_d$ from $\mathcal{C}_{d-2}$. To that end we can consider the following map. Let $d \geq 2$ and define $\psi_d: \mathcal{C}_{d-2} \times \mathbb{R}_{\geq 0} \to \mathcal{B}_d$ such that
\[
\psi_d(f,t)(x) = f(x) \cdot (x-t)^2.
\]
We will now check the properties of this map.

\begin{lemma}[Base Boundary Surjective]
\label{lem:BaseBoundarySurjective}
Let $d \geq 2$. Then $\psi_d:\mathcal{C}_{d-2} \times \mathbb{R}_{\geq 0} \to \mathcal{B}_d$ is surjective.
\end{lemma}

\begin{proof}
	Let $g \in \mathcal{B}_d$. Then by definition of the base boundary $g$ takes the form
	\[
		g = (x-z_1)(x-z_2)\cdots(x-z_{d-2})(x-t)^2
	\]
	for some $t \geq 0$. Now it suffices to show that $f = (x-z_1)\cdots(x-z_{d-2})$ is an element of $\mathcal{C}_{d-2}$. Clearly $f$ is a monic, degree $d-2$ polynomial with real coefficients. By means of contradiction assume that $f(y) < 0$ for some $y \geq 0$. Then
	\begin{align*}
		g(y) &= f(y)(y-t)^2 \\
		&\leq 0.
	\end{align*}
	If $(y-t)^2 > 0$, then we have our contradiction. So now assume that $(y-t)^2 = 0$. In other words, $t = y$. Since $f$ is continuous there exists some $\varepsilon > 0$ such that $f(y+\varepsilon) < 0$. Then 
	\begin{align*}
		g(y+\varepsilon) &= f(y+\varepsilon)(y+\varepsilon-t)^2 \\
		&= f(y+\varepsilon)\varepsilon^2 \\
		&< 0.
	\end{align*}
	Again, we arrive at a contradiction.
\end{proof}

\begin{lemma}[Base Boundary Almost Injective]
	\label{lem:BaseBoundaryInjective}
	Let $d \geq 2$. Then $\psi_d:\mathcal{C}_{d-2} \times \mathbb{R}_{\geq 0} \to \mathcal{B}_d$ is almost injective.
\end{lemma}

\begin{proof}
	Note that $\mathcal{B}_{d-2}$ is part of the boundary of $\mathcal{C}_{d-2}$, and thus it is a measure zero subset. This also means that $\mathcal{B}_{d-2} \times \mathbb{R}_{\geq 0} \subset \mathcal{C}_{d-2} \times \mathbb{R}_{\geq 0}$ is a measure zero subset. We will avoid this subset. By means of contradiction let $f,g \in \mathcal{C}_{d-2} \setminus \mathcal{B}_{d-2}$ and $t,s \in \mathbb{R}_{\geq 0}$ such that $f \neq g$ and $\psi_d(f,t) = \psi_d(g,s)$. 
	
	One case is when $t=s$. If this is true, then we have $f(x)\cdot(x-t)^2 = g(x) \cdot (x-s)^2$, which implies $f(x) = g(x)$. This contradicts our assumption. Now we can assume that $t \neq s$. From the fundamental theorem of algebra we know that $\psi_d(f,t)$ and $\psi_d(g,s)$ have the same set of complex roots. Using this along with $t \neq s$ we have
	\[
		\psi_d(f,t) = \psi_d(g,s) = (x-z_1)\cdots(x-z_{d-4})(x-t)^2(x-s)^2.
	\]
	Hence, $g = (x-z_1)\cdots(x-z_{d-4})(x-t)^2$. However, this implies that $g \in \mathcal{B}_{n-2}$ since $g(t) = g'(t) = 0$. Again, this contradicts our assumption.
\end{proof}

Note that we need to go up to $d = 4$ to begin to see injectivity failing. The main issue algebraically is that when $d \geq 4$ a polynomial can have two different critical points which are also non-negative roots. Luckily, as long as the original polynomial doesn't have any roots with multiplicity 2, this issue disappears. 
\begin{example}[Injectivity Failing]
	\[
		\psi_4\left((x-1)^2,2\right) = (x-1)^2(x-2)^2 = \psi_4\left((x-2)^2,1 \right)
	\] 
\end{example}
\section{Extending the Base Boundary}
\label{sec:Ext}
In the case of $\mathcal{C}_2$ we had very limited options for directions to extend $\mathcal{B}_2$. By looking at Figure \ref{fig:C2}, we see that the only direction that works is extending $\mathcal{B}_2$ in the direction of the linear term. However, once the degree $d$ grows we start to get many more directions we could go. If we were to take naive guesses at the direction to extend, I believe there are two very natural choices given $\mathcal{C}_2$. We could view the direction we extended $\mathcal{B}_2$ to get $\mathcal{C}_2$ as either the direction of the linear term or as the direction of the term which is one degree less than $d$. Each of these give us a general approach for any $d$.

We believe that always using the linear term to extend will yield the simplest parametrization. So that is what we will use. However, there is freedom here to possibly choose a different direction. This idea suggests that we consider the map $\phi_d: \mathcal{B}_d \times \mathbb{R}_{\geq 0} \to \mathcal{C}_d$ defined by
\[
	\phi_d(f,t)(x) = f(x) + tx.
\]
Just as before, we will now check the properties of this map.

\begin{lemma}[Extension Surjective]
\label{lem:ExtendSurjective}
Let $d \geq 2$. Then $\phi_d: \mathcal{B}_d \times \mathbb{R}_{\geq 0} \to \mathcal{C}_d$ is surjective.
\end{lemma}

\begin{proof}
	Let $g \in \mathcal{C}_d$. Note that $\phi_d$ is surjective if we can find a $t \in \mathbb{R}_{\geq 0}$ such that $g(x) - tx \in \mathcal{B}_d$. To that end let
	\[
		t = \inf\left\{\frac{g(x)}{x} : x > 0\right\}.
	\]
	Now we must show that $g(x) - tx \in \mathcal{C}_d$ and for some $y \geq 0$ we have $g(y) - ty = g'(y) - t = 0$. Both of these together would give us that $g(x) - tx \in \mathcal{B}_d$.
	
	For the first claim assume that $g(x) - tx \notin \mathcal{C}_d$. Thus, there exists some $y \geq 0$ such that $g(y) - ty < 0$. If $y = 0$, then $g(0) < 0$. This contradicts $g \in \mathcal{C}_d$. So we can assume that $y > 0$. Then by rearranging $g(y) - ty < 0$ we have
	\[
		t > \frac{g(y)}{y}.
	\]
	This contradicts $t$ being the infimum of the set above. So we can conclude that $g(x) - tx \in \mathcal{C}_d$.
	
	The second claim needs to be split into two cases. For the first case assume that $t \in \left\{\frac{g(x)}{x} : x>0 \right\}$. In other words, $t = \min\left\{\frac{g(x)}{x} : x > 0\right\}$. In this case there exists a $y > 0$ such that $t = \frac{g(y)}{y}$. Rearranging this equation gives us $g(y)-ty = 0$. Also, since $y$ minimizes the function $\frac{g(x)}{x}$ over an open set, we must have that $y$ causes the derivative of $\frac{g(x)}{x}$ to vanish. To be explicit,
	\[
		\frac{yg'(y) - g(y)}{y^2} = 0.
	\]
	If we rearrange this equation we get $g'(y) = \frac{g(y)}{y} = t$. In other words, $g'(y) - t = 0$.
	
	Now we must consider the case when the set $\left\{\frac{g(x)}{x} : x > 0\right\}$ has no minimum. In this case we have either
	\[
		t = \lim\limits_{x \to \infty}\frac{g(x)}{x}
	\]
	or
	\[
		t = \lim\limits_{x \to 0}\frac{g(x)}{x}.
	\]
	Since the degree of $g$ is at least 2 the limit in the first option doesn't exist. Thus,
	\[
		t = \lim\limits_{x \to 0}\frac{g(x)}{x}.
	\]
	The only way for this limit to exist is if $g$ factors as $g(x) = x\cdot h(x)$ for some $h\in \mathcal{C}_{d-1}$. Then we have $g(0) - t\cdot 0 = 0\cdot h(0) - t\cdot 0 = 0$. Also, by applying L'H\^opital's rule we have
	\[
		t = \lim\limits_{x \to 0}\frac{g(x)}{x} = \lim\limits_{x \to 0} \left(x \cdot h(x)\right)'
		= \lim\limits_{x \to 0} \left(h(x) + x\cdot h'(x)\right) = h(0).
	\]
	Using this we have $g'(0) - t = h(0) - h(0) = 0$. Thus, we can conclude that $g(x) - tx \in \mathcal{B}_d$.
\end{proof}

\begin{lemma}[Extension Injective]
	\label{lem:ExtendInjective}
	Let $d \geq 2$. Then $\phi_d: \mathcal{B}_d \times \mathbb{R}_{\geq 0} \to \mathcal{C}_d$ is injective.
\end{lemma}

\begin{proof}
	Let $f,g \in \mathcal{B}_d$ and $t,s \in \mathbb{R}_{\geq 0}$ such that $\phi_d(f,t) = \phi_d(g,s)$. Clearly if $t = s$, then we have $f = g$, and the map is injective. So without loss of generality we will assume that $t > s$ and eventually arrive at a contradiction. Note that we have the equation
	\begin{equation}\label{eq:inj1}
		f(x) - g(x) + (t-s)x = 0
	\end{equation}
	which holds for all $x \in \mathbb{R}$. Since $g \in \mathcal{B}_d$ we have some $y \geq 0$ such that $g(y) = g'(y) = 0$. Substituting $y$ into equation (\ref{eq:inj1}) we have
	\[
		f(y) + (t-s)y = 0.
	\]
	Note that both $f(y)$ and $(t-s)y$ are non-negative. Then we can conclude $f(y) = 0$ and $(t-s)y = 0$. By assumption $t-s > 0$, and so $y = 0$. Now if we differentiate equation (\ref{eq:inj1}) and evaluate at $x = y = 0$ we get
	\[
		f'(y) - g'(y) + (t-s) = f'(0) + (t-s) = 0.
	\]
	Thus, $f$ is decreasing at $x=0$. But this along with the fact that $f(0) = 0$ implies that $f$ is not copositive. This contradicts $f$ being an element of $\mathcal{B}_d$.
\end{proof}
\section{Combining the Base Boundary with the Extension}\label{sec:combo}

Now we will combine the previous results to arrive at a recursive way to parametrize $\mathcal{C}_d$. Since the recursion jumps by $2$ every time we will need two different starting points depending on the parity of $d$. When $d$ is even we will start at $\mathcal{C}_0$, which consists of the single polynomial $f(x) = 1$. When $d$ is odd we will begin with $\mathcal{C}_1 = \{x+a_0 : a_0 \geq 0\}$. Regardless, we get the following string of compositions
\[
	\Phi_{2k}(t_0,\dots,t_{2k-1}) = \phi_{2k}(\psi_{2k}(\dots\phi_4(\psi_4(\phi_2(\psi_2(1,t_0),t_1),t_2),t_3),\dots,t_{2k-2}),t_{2k-1}),
\]
or
\[
	\Phi_{2k+1}(t_0,\dots,t_{2k}) = \phi_{2k+1}(\psi_{2k+1}(\dots\phi_5(\psi_5(\phi_3(\psi_3(x,t_0),t_1),t_2),t_3),\dots,t_{2k-1}),t_{2k}).
\]
This idea is the map in Theorem \ref{thm:main}. In words we can describe this map as taking a copositive polynomial, adding a critical point somewhere on the positive $x$-axis, linearly shifting the resulting polynomial, and then repeating this process until the desired degree is achieved. We can see what's happening geometrically in Figure \ref{fig:map}.

\begin{figure}[h!]
	\centering
	\includegraphics[scale=0.75]{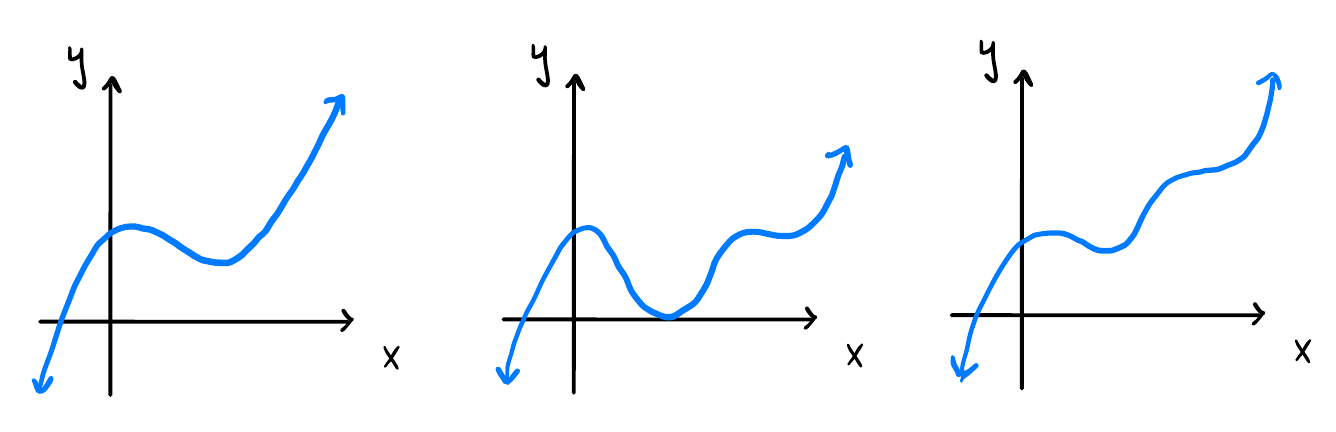}
	\caption{Arbitrary polynomial from $\mathcal{C}_3$ (left), creating critical point with $\psi_5$ (center), linearly shifting the center polynomial with $\phi_5$ (right)}
	\label{fig:map}
\end{figure}

\begin{proof}(Theorem \ref{thm:main})
	We will begin with $\Phi_0$ and $\Phi_1$. Technically speaking $\Phi_0$ is not bijective, but we can say that is if we allow ourselves to claim that the ``empty element" maps to the only element in the codomain. Regardless, this is a small detail, and it doesn't affect the recursion. Clearly $\Phi_1$ is bijective since it is essentially the identity map.
	
	Now let $d \geq 2$ and consider $\Phi_d$. For a proof by induction assume that $\Phi_k$ is surjective and generically injective for all $k < d$. We will now show that $\Phi_d: (\mathbb{R}_{\geq 0})^d \to \mathcal{C}_d$ is surjective and generically injective. Note that using the maps $\psi_d$ and $\phi_d$ defined in sections \ref{sec:BB} and \ref{sec:Ext}, we can write $\Phi_d$ as follows.
	\[
		\Phi_d(t_0,\dots,t_{d-1}) = \phi_d\Big(\ \psi_d\big(\Phi_{d-2}(\mathbf{t}_{d-2}),\ t_{d-2}\big),\ t_{d-1}\Big)
	\]
	Let $f \in \mathcal{C}_d$. For surjectivity we must now find an element $\mathbf{t}_d \in (\mathbb{R}_{\geq 0})^d$ which maps to $f$. From Lemmas \ref{lem:ExtendSurjective} and \ref{lem:ExtendInjective} we know that $\phi_d$ is bijective. Thus, there is a pair $(g,t_{d-1}) \in \mathcal{B}_d \times \mathbb{R}_{\geq 0}$ such that $\phi_d(g,t_{d-1}) = f$. From Lemma $\ref{lem:BaseBoundarySurjective}$ we have that $\psi_d$ is surjective. Thus, there exists a pair $(h,t_{d-2}) \in \mathcal{C}_{d-2} \times \mathbb{R}_{\geq 0}$ such that $\psi_d(h,t_{d-2}) = g$. Hence, $f = \phi_d\left( \psi_d(h,t_{d-2}),t_{d-1}\right)$. We have $(t_0,\dots,t_{d-3}) \in (\mathbb{R}_{\geq 0})^{d-2}$ such that $\Phi_{d-2}(t_0,\dots,t_{d-3}) = h$ by the induction hypothesis. This gives us surjectivity.
	
	For generic injectivity assume that $\Phi_d(\mathbf{t}_d) = \Phi_d(\mathbf{t}'_d)$. Again, using the maps $\phi_d$ and $\psi_d$ we have
	\[
		\phi_d\Big(\ \psi_d\big(\Phi_{d-2}(\mathbf{t}_{d-2}),\ t_{d-2}\big),\ t_{d-1}\Big) =
		\phi_d\Big(\ \psi_d\big(\Phi_{d-2}(\mathbf{t}'_{d-2}),\ t'_{d-2}\big),\ t'_{d-1}\Big).
	\]
	Since $\phi_d$ is bijective the above equation simplifies to
	\[
	\bigg(\psi_d\big(\Phi_{d-2}(\mathbf{t}_{d-2}),\ t_{d-2}\big),\ t_{d-1}\bigg) =
	\bigg(\psi_d\big(\Phi_{d-2}(\mathbf{t}'_{d-2}),\ t'_{d-2}\big),\ t'_{d-1}\bigg).
	\]
	So, $t_{d-1} = t'_{d-1}$ and $\psi_d\big(\Phi_{d-2}(\mathbf{t}_{d-2}),\ t_{d-2}\big) = \psi_d\big(\Phi_{d-2}(\mathbf{t}'_{d-2}),\ t'_{d-2}\big)$. From Lemma \ref{lem:BaseBoundaryInjective} we have that $\psi_d$ is generically injective. Therefore we can change $\mathbf{t}_d$ and $\mathbf{t}'_d$ so that $\psi_d$ has a unique inverse, and we have equality on the ordered pair
	\[
		\big(\Phi_{d-2}(\mathbf{t}_{d-2}),\ t_{d-2}\big) =
		\big(\Phi_{d-2}(\mathbf{t}'_{d-2}),\ t'_{d-2}\big).
	\]
	Now we have that $t_{d-2} = t'_{d-2}$ and $\Phi_{d-2}(\mathbf{t}_{d-2}) = \Phi_{d-2}(\mathbf{t}'_{d-2})$. Here the induction hypothesis takes over, and we have generic injectivity.
\end{proof}

\section{Conclusion}
We hope that the results in this paper will help lay the foundation for further parametrizations of semialgebraic sets. Of course the next big hurdle will be copositive polynomials in two variables. However, as we have seen in several algebraic contexts, making the jump to multiple variables tends to require new techniques and ideas. Thus, we expect that the two-variable case will not arise simply from our results. Having said that, one thing to note is that the idea of the base boundary can easily be extended to multiple variables, and we believe that this object could hold an important role in any future parametrizations.

\bigskip \noindent \textbf{Acknowledgements.}
Hoon Hong was partially supported by US National Science Foundation NSF-CCF-2212461.

\bibliographystyle{plain}
\bibliography{biblio}

\end{document}